\documentclass[12pt,reqno]{article}
\usepackage[usenames]{color}
\usepackage[colorlinks=true,linkcolor=webgreen,filecolor=webbrown,citecolor=webgreen]{hyperref}
\usepackage{amssymb}
\usepackage{amsmath}
\usepackage{amsfonts}
\usepackage[T1]{fontenc}
\usepackage{enumerate}
\usepackage{graphicx}
\usepackage{color}

\definecolor{webgreen}{rgb}{0,.5,0}
\definecolor{webbrown}{rgb}{.6,0,0}
\newtheorem{theorem}{Theorem}

\newtheorem{corollary}[theorem]{Corollary}

\newtheorem{lemma}[theorem]{Lemma}

\newenvironment{proof}[1][Proof]{\noindent \textbf{#1.} }{\  \rule{0.5em}{0.5em}}

\allowdisplaybreaks

\begin{document}

\begin{center}
\vskip1cm

{\LARGE \textbf{A complete asymptotic expansion for the semi-exponential
Post--Widder operators}}

\vspace{2cm}

{\large Ulrich Abel}\\[0pt]
\textit{Technische Hochschule Mittelhessen}\\[0pt]
\textit{Fachbereich MND}\\[0pt]
\textit{Wilhelm-Leuschner-Stra\ss e 13, 61169 Friedberg }\\[0pt]
\textit{Germany}\\[0pt]
\href{mailto:Ulrich.Abel@mnd.thm.de}{\texttt{Ulrich.Abel@mnd.thm.de}}\\[0pt]
\textit{ORCID: 0000-0003-1889-4850}

\vspace{1cm}

{\large Octavian Agratini}\\[0pt]
\textit{Tiberiu Popoviciu Institute of Numerical Analysis}\\[0pt]
\textit{Romanian Academy}\\[0pt]
\textit{Street F\^{a}nt\^{a}nele 57, 400320 Cluj-Napoca}\\[0pt]
\textit{Romania}\\[0pt]
\href{mailto:agratini@ictp.acad.ro}{\texttt{agratini@ictp.acad.ro}}\\[0pt]
\textit{ORCID: 0000-0002-2406-4274}

\vspace{1cm}

{\large Radu P\u{a}lt\u{a}nea}\\[0pt]
\textit{Transilvania University}\\[0pt]
\textit{Department of Mathematics}\\[0pt]
\textit{Street Eroilor 29, 500036 Bra\c{s}ov}\\[0pt]
\textit{Romania}\\[0pt]
\href{mailto:radupaltanea@yahoo.com}{\texttt{radupaltanea@yahoo.com}}\\[0pt]
\textit{ORCID: 0000-0002-9923-4290}
\end{center}

\vspace{1cm}

{\large \textbf{Abstract.}}

\bigskip

In the present paper, we study the asymptotic properties of the
semi-exponential Post--Widder operator. It is connected with $p\left(
x\right) =x^{2}$. The main result is a pointwise complete asymptotic
expansion valid for locally smooth functions of exponential growth. All
coefficients are derived and explicitly given. As a special case we recover
the complete asymptotic expansion for the classical Post--Widder operator.
\bigskip


 \vspace{2cm}


\section{Introduction}

\label{intro}

Many prominent approximation operators are exponential-type operators. They
were firstly considered by Ismail and May \cite{Ismail-May-JMAA-1978} in
1978. The exponential-type operators preserve the linear functions. Tyliba
and Wachnicki \cite{Tyliba-Wachnicki-2005} extended the definition of Ismail
and May \cite{Ismail-May-JMAA-1978} by proposing a more general family of
operators. For a non-negative real number $\beta $, they introduced the
operators $L_{\lambda }^{\beta }$. For $\beta >0$, they are not of
exponential type but similar to exponential-type operators. Recently, Herzog 
\cite{Herzog-Monika-Symmetry-2021} further extended the studies and termed
such operators as semi-exponential type operators. An operator of the form 
\begin{equation*}
(L_{n}^{\beta }f)\left( x\right) =\int_{I}W_{\beta }\left( n,x,t\right)
f\left( t\right) dt\text{\  \qquad }\left( x\in I\right) ,
\end{equation*}%
is called a semi-exponential operator if its kernel $W_{\beta }\left(
n,x,t\right) $ satisfies the differential equation 
\begin{equation*}
\frac{\partial }{\partial x}W_{\beta }\left( n,x,t\right) =\left( \frac{%
n\left( t-x\right) }{p\left( x\right) }-\beta \right) W_{\beta }\left(
n,x,t\right) .
\end{equation*}%
The function $p$ is supposed to be positive on the corresponding interval $I$%
. In particular, for $\beta >0$, one has $L_{n}^{\beta }e_{1}\neq e_{1}$,
where $e_{r}\left( t\right) =t^{r}$\ $\left( r=0,1,2,...\right) $. In the
case $\beta =0$, the operator $L_{n}^{\beta =0}$ reduces to the
exponential-type operator studied by Ismail and May \cite%
{Ismail-May-JMAA-1978}. A collection of such operators can be found in the
recent book \cite[Ch. 1]{Gupta-Rassias-book-2021}.

For $A>0$, let $E_{A}\left[ 0,\infty \right) $ denote the space of all
locally integrable functions $f$ on $\left[ 0,\infty \right) $, which
satisfy the growth condition $\left \vert f\left( t\right) \right \vert \leq
Ce^{At}$ $\left( t\geq 0\right) $, for some constant $C>0$. Set $E\left[
0,\infty \right) =\bigcup_{A>0}E_{A}\left[ 0,\infty \right) $.

The semi-exponential Post--Widder operators \cite[Eq. (10)]%
{Herzog-Monika-Symmetry-2021} (cf. \cite[Sect. 2.4]%
{Abel-Gupta-Sisodia-RACSAM-2022} and \cite[Page 3]%
{Herzog-Monika-Symmetry-2023}) are defined, for $x>0$, by 
\begin{equation}
\left( P_{n}^{\beta }f\right) \left( x\right) =e^{-\beta x}\left( \frac{n}{x}%
\right) ^{n}\sum_{\nu =0}^{\infty }\frac{\left( n\beta \right) ^{\nu }}{\nu !%
}\frac{1}{\Gamma \left( n+\nu \right) }\int_{0}^{\infty }t^{n+\nu
-1}e^{-nt/x}f\left( t\right) dt.  \label{def-P-usual-form}
\end{equation}%
They are well-defined for all functions $f$\ of (at most) polynomial growth.
Moreover, $\left( P_{n}^{\beta }f\right) \left( x\right) $ is well-defined,
for all functions $f\in E\left[ 0,\infty \right) $, if $n>Ax$ (see Lemma~\ref%
{lemma-exp} below). The operators $\left( \ref{definition-P}\right) $ are
connected with the function $p\left( x\right) =x^{2}$, have the kernel 
\begin{equation}
W_{\beta }\left( n,x,t\right) =e^{-\beta x-nt/x}\left( \frac{n}{x}\right)
^{n}\sum_{\nu =0}^{\infty }\frac{\left( n\beta \right) ^{\nu }}{\nu !\Gamma
\left( n+\nu \right) }t^{n+\nu -1}  \label{kernel-W}
\end{equation}%
and satisfy the equation 
\begin{equation}
\left( P_{n}^{\beta }f\right) ^{\prime }\left( x\right) =nx^{-2}\left(
P_{n}^{\beta }\left( \psi _{x}f\right) \right) \left( x\right) -\beta \left(
P_{n}^{\beta }f\right) \left( x\right) ,  \label{ode-P}
\end{equation}%
where $\psi _{x}\left( t\right) =t-x$. In the special case $\beta =0$, we
recover the classical Post--Widder operators 
\begin{equation*}
\left( P_{n}^{\beta =0}f\right) \left( x\right) =\frac{1}{\Gamma \left(
n\right) }\left( \frac{n}{x}\right) ^{n}\int_{0}^{\infty
}t^{n-1}e^{-nt/x}f\left( t\right) dt.
\end{equation*}%
Observing that 
\begin{equation*}
n^{n}\sum_{\nu =0}^{\infty }\frac{\left( n\beta \right) ^{\nu }}{\nu !\Gamma
\left( \nu +n\right) }t^{\nu +n-1}=n\left( nt/\beta \right) ^{\left(
n-1\right) /2}I_{n-1}\left( 2\sqrt{n\beta t}\right) ,
\end{equation*}%
where $I_{n}$ denotes the modified Bessel function of the first kind,
defined by the series expansion 
\begin{equation*}
I_{n}\left( z\right) =\sum_{\nu =0}^{\infty }\frac{\left( z/2\right) ^{2\nu
+n}}{\nu !\Gamma \left( \nu +n+1\right) },
\end{equation*}%
we obtain the alternative representation 
\begin{equation*}
\left( P_{n}^{\beta }f\right) \left( x\right) =nx^{-n}e^{-\beta
x}\int_{0}^{\infty }\left( \frac{nt}{\beta }\right) ^{\left( n-1\right)
/2}e^{{-nt}/{x}}I_{n-1}\left( 2\sqrt{n\beta t}\right) f\left( t\right) dt.
\end{equation*}%
We mention the following relation between the operator $P_{n}^{\beta }$ and
the classical Post--Widder operator, i.e., its special instance $%
P_{n}^{\beta =0}$: 
\begin{equation*}
\left( P_{n}^{\beta }f\right) \left( x\right) =e^{-\beta x}\left(
P_{n}^{\beta =0}\left( _{0}F_{1}\left( ;n;n\beta \cdot \right) f\right)
\right) \left( x\right) ,
\end{equation*}%
where $_{0}F_{1}$ denotes the hypergeometric function 
\begin{equation*}
_{0}F_{1}\left( ;n;z\right) =\Gamma \left( n\right) \sum_{\nu =0}^{\infty }%
\frac{z^{\nu }}{\nu !\Gamma \left( n+\nu \right) }=\frac{\Gamma \left(
n\right) }{z^{\left( n-1\right) /2}}I_{n-1}\left( 2\sqrt{z}\right) .
\end{equation*}%
We prefer the equivalent form 
\begin{equation}
\left( P_{n}^{\beta }f\right) \left( x\right) =e^{-\beta x}\sum_{\nu
=0}^{\infty }\frac{\left( \beta x\right) ^{\nu }}{\nu !\Gamma \left( n+\nu
\right) }\int_{0}^{\infty }t^{n+\nu -1}e^{-t}f\left( \frac{xt}{n}\right) dt,
\label{definition-P}
\end{equation}%
which follows from $\left( \ref{def-P-usual-form}\right) $ by a change of
variable. It includes $\left( P_{n}^{\beta }f\right) \left( 0\right)
=f\left( 0\right) $, i.e., interpolation of the function at the point $x=0$.

The operators $\left( \ref{definition-P}\right) $ were studied by Grewal and
Rani \cite{Grewal-Rani-Ferrara-2024} and by Kumar and Deo \cite%
{Kumar-Deo-Miskolc-2025}. Among other approximation properties they gave the
Voronovskaja-type formula (\cite[Theorem~3 (ii)]{Grewal-Rani-Ferrara-2024}
and \cite[Theorem~6]{Kumar-Deo-Miskolc-2025}) 
\begin{equation*}
\lim_{n\rightarrow \infty }n\left( \left( P_{n}^{\beta }f\right) \left(
x\right) -f\left( x\right) \right) =x^{2}\left( \beta f^{\prime }\left(
x\right) +\frac{1}{2}f^{\prime \prime }\left( x\right) \right) ,
\end{equation*}%
provided that $f^{\prime \prime }\left( x\right) $ exists (the formula in 
\cite[Theorem~6]{Kumar-Deo-Miskolc-2025} is misprinted as the factor $1/2$
is missing). Furthermore, Kumar and Deo did not specify the domain of
functions $f$ for which the limit is valid. Moreover, Grewal and Rani \cite[%
Theorems~3 and 4]{Grewal-Rani-Ferrara-2024} proved quantitative versions of
the Voronovskaja-type formula in polynomial weighted spaces via a weighted
modulus of continuity. 

In this note we derive a complete asymptotic expansion of the operators $%
P_{n}^{\beta }$ as $n$ tends to infinity. We prove that, for locally
integrable functions $f$ satisfying certain growth conditions, the operators 
$P_{n}^{\beta }$ possess, for a given point $x>0$, the complete asymptotic
expansion 
\begin{equation*}
\left( P_{n}^{\beta }f\right) \left( x\right) \sim f\left( x\right)
+\sum_{k=1}^{\infty }\frac{1}{n^{k}}c_{k}^{\beta }\left( f,x\right) \text{
\qquad }\left( n\rightarrow \infty \right) ,
\end{equation*}%
provided that the function $f$ is sufficiently smooth at $x$. The latter
relation means that, for all positive integers $q$, 
\begin{equation*}
\left( P_{n}^{\beta }f\right) \left( x\right) =f\left( x\right)
+\sum_{k=1}^{q}\frac{1}{n^{k}}c_{k}^{\beta }\left( f,x\right) +o\left(
n^{-q}\right) \text{ \qquad }\left( n\rightarrow \infty \right) .
\end{equation*}%
All coefficients $c_{k}^{\beta }\left( f,x\right) $, which are independent
of $n$, will be derived in an explicit form.

\section{The complete asymptotic expansion for the operators $P_{ 
\lowercase{n}}^{\protect \beta }$}

In this section we derive a pointwise asymptotic expansion for the sequence $%
\left( \left( P_{n}^{\beta }f\right) \left( x\right) \right) _{n=1}^{\infty
} $ as $n\rightarrow \infty $, for functions $f$ which are sufficiently
smooth at the point $x$. It turns out that the coefficients contain Stirling
numbers of the first kind. The unsigned Stirling numbers of the first kind $%
\left[ 
\begin{array}{c}
r \\ 
\ell%
\end{array}%
\right] $ are defined by 
\begin{equation}
z^{\underline{r}}=\sum_{\ell =0}^{r}\left( -1\right) ^{r-\ell }\left[ 
\begin{array}{c}
r \\ 
\ell%
\end{array}%
\right] z^{\ell }\text{\  \qquad }\left( r=0,1,2,\ldots \right) ,
\label{definition-Stirling-numbers}
\end{equation}%
where $z^{\underline{0}}=1$, $z^{\underline{r}}=z\left( z-1\right) \cdots
\left( z-r+1\right) $, for $r\in \mathbb{N}$, denote the falling factorials.
It is easily be seen that their generating function is given by 
\begin{equation}
\left( z+r-1\right) ^{\underline{r}}=\sum_{\ell =0}^{r}\left[ 
\begin{array}{c}
r \\ 
\ell%
\end{array}%
\right] z^{\ell }  \label{Stirling-number-GF}
\end{equation}%
For $z\in \mathbb{R}$ and $m\in \mathbb{N}_{0}$, the binomial coefficient $%
\binom{z}{m}$ is given by $\binom{z}{m}=z^{\underline{m}}/m!$ such that $%
\binom{r}{m}=0$, for all integers $r,m$ with $m>r\geq 0$. We recall that the
Stirling numbers of the first kind possess the representation 
\begin{equation}
\left[ 
\begin{array}{c}
r \\ 
r-m%
\end{array}%
\right] =\sum_{\ell =m}^{2m}s_{2}\left( \ell ,\ell -m\right) \binom{r}{\ell }%
=\sum_{\ell =0}^{m}s_{2}\left( \ell +m,\ell \right) \binom{r}{\ell +m}\text{ 
}\hspace{0.5cm}\left( 0\leq m\leq r\right)
\label{associated Stirling numbers of the first kind}
\end{equation}%
(see \cite[page 226, Ex. 16]{Comtet-Advanced Combinatorics-1974}). The
coefficients $s_{2}\left( \ell ,\ell -m\right) $, called associated Stirling
numbers of the first kind, are independent of $r$. The associated Stirling
numbers of the first kind can be defined by their double generating function 
\begin{equation}
\sum_{i,j=0}^{\infty }s_{2}\left( i,j\right) \frac{t^{i}}{i!}%
u^{j}=e^{-tu}\left( 1+t\right) ^{u}  \label{associated Stirling number-GF}
\end{equation}%
(see \cite[page 295, Ex. *20]{Comtet-Advanced Combinatorics-1974}).

Throughout the paper we define, for integers $k,s,j\geq 0$, 
\begin{equation}
a\left( k,s,j\right) =\sum_{i=j}^{\min \left \{ k,2k-s\right \} }\binom{i-1}{%
i-j}\frac{s!}{\left( s-i\right) !}s_{2}\left( s-i,s-k\right)
\label{def-coefficient-a(k,s,j)}
\end{equation}%
with the convention that a sum is to be read as zero if the lower index
exceeds the upper index.

The following theorem presents as our main result the pointwise complete
asymptotic expansion for the operators $P_{n}^{\beta }$ as $n\rightarrow
\infty $.

\begin{theorem}
\label{theorem-expansion}Let $q\in \mathbb{N}$ and $x\in \left( 0,\infty
\right) $. For each function $f\in E\left[ 0,\infty \right) $, which admits
the derivative $f^{\left( 2q\right) }\left( x\right) $, the operators $%
P_{n}^{\beta }$ possess the asymptotic expansion 
\begin{equation*}
\left( P_{n}^{\beta }f\right) \left( x\right) =f\left( x\right)
+\sum_{k=1}^{q}\frac{1}{n^{k}}c_{k}^{\beta }\left( f,x\right) +o\left(
n^{-q}\right) \text{ }\qquad \left( n\rightarrow \infty \right) ,
\end{equation*}%
where the coefficients $c_{k}^{\beta }\left( f,x\right) $ are given by 
\begin{equation*}
c_{k}^{\beta }\left( f,x\right) =\sum_{s=k}^{2k}\frac{f^{\left( s\right)
}\left( x\right) }{s!}x^{s}\sum_{j=0}^{k}a\left( k,s,j\right) \frac{\left(
\beta x\right) ^{j}}{j!}
\end{equation*}%
and $a\left( k,s,j\right) $ is defined in $\left( \ref%
{def-coefficient-a(k,s,j)}\right) $.
\end{theorem}

As an immediate corollary we obtain, for $f\in E\left[ 0,\infty \right) $,
the Voronovskaja-type formula 
\begin{equation*}
\lim_{n\rightarrow \infty }n\left( \left( P_{n}^{\beta }f\right) \left(
x\right) -f\left( x\right) \right) =\beta x^{2}f^{\prime }\left( x\right) +%
\frac{1}{2}x^{2}f^{\prime \prime }\left( x\right) ,
\end{equation*}%
provided that $f^{\prime \prime }\left( x\right) $ exists\ (cf. \cite[%
Theorem~6]{Kumar-Deo-Miskolc-2025}).

For the convenience of the reader we list some initial coefficients of the
complete asymptotic expansion in Theorem~\ref{theorem-expansion}: 
\begin{eqnarray*}
c_{1}^{\beta }\left( f,x\right) &=&\beta x^{2}f^{\prime }\left( x\right) +%
\frac{1}{2}x^{2}f^{\prime \prime }\left( x\right) , \\
c_{2}^{\beta }\left( f,x\right) &=&\left( \beta x^{3}+\frac{1}{2}\beta
^{2}x^{4}\right) f^{\left( 2\right) }\left( x\right) +\left( \frac{1}{3}%
x^{3}+\frac{1}{2}\beta x^{4}\right) f^{\left( 3\right) }\left( x\right) +%
\frac{1}{8}x^{4}f^{\left( 4\right) }\left( x\right) , \\
c_{3}^{\beta }\left( f,x\right) &=&\left( \beta x^{4}+\beta ^{2}x^{5}+\frac{1%
}{6}\beta ^{3}x^{6}\right) f^{\left( 3\right) }\left( x\right) +\frac{1}{12}%
\left( 3x^{4}+10\beta x^{5}+3\beta ^{2}x^{6}\right) f^{\left( 4\right)
}\left( x\right) \\
&&+\frac{1}{24}\left( 4x^{5}+3\beta x^{6}\right) f^{\left( 5\right) }\left(
x\right) +\frac{1}{48}x^{6}f^{\left( 6\right) }\left( x\right) , \\
c_{4}^{\beta }\left( f,x\right) &=&\left( \beta x^{5}+\frac{3}{2}\beta
^{2}x^{6}+\frac{1}{2}\beta ^{3}x^{7}+\frac{1}{24}\beta ^{4}x^{8}\right)
f^{\left( 4\right) }\left( x\right) \\
&&+\frac{1}{60}\left( 12x^{5}+65\beta x^{6}+40\beta ^{2}x^{7}+5\beta
^{3}x^{8}\right) f^{\left( 5\right) }\left( x\right) \\
&&+\frac{1}{144}\left( 26x^{6}+42\beta x^{7}+9\beta ^{2}x^{8}\right)
f^{\left( 6\right) }\left( x\right) \\
&&+\frac{1}{48}\left( 2x^{7}+\beta x^{8}\right) f^{\left( 7\right) }\left(
x\right) +\frac{1}{384}x^{8}f^{\left( 8\right) }\left( x\right) .
\end{eqnarray*}%
In the special case $\beta =0$, we obtain the complete asymptotic expansion
for the classical Post--Widder operators $P_{n}^{\beta =0}$. Observing that $%
\binom{i-1}{i}=0$, for each positive integer $i$, we see that 
\begin{equation*}
a\left( k,s,j=0\right) =\sum_{i=0}^{\min \left \{ k,2k-s\right \} }\binom{i-1}{%
i}\frac{s!}{\left( s-i\right) !}s_{2}\left( s-i,s-k\right) =s_{2}\left(
s,s-k\right) .
\end{equation*}%
Thus, we get the following result which was derived in \cite%
{Abel-Gupta-AIOT-2023}.

\begin{corollary}
\label{corollary-expansion-Post-Widder}Let $q\in \mathbb{N}$ and $x\in
\left( 0,\infty \right) $. For each function $f\in E\left[ 0,\infty \right) $%
, which admits the derivative $f^{\left( 2q\right) }\left( x\right) $, the
classical Post--Widder operators $P_{n}^{\beta =0}$ possess the asymptotic
expansion 
\begin{equation*}
\left( P_{n}^{\beta =0}f\right) \left( x\right) =f\left( x\right)
+\sum_{k=1}^{q}\frac{1}{n^{k}}\sum_{s=k}^{2k}s_{2}\left( s,s-k\right) \frac{%
f^{\left( s\right) }\left( x\right) }{s!}x^{s}+o\left( n^{-q}\right) \text{ }%
\qquad \left( n\rightarrow \infty \right) ,
\end{equation*}%
where $s_{2}\left( s,s-k\right) $ are the associated Stirling numbers of the
first kind, defined in $\left( \ref{associated Stirling number-GF}\right) $.
\end{corollary}

\section{Auxiliary results and proofs}

Let $e_{r}$ $\left( r\in \mathbb{N}_{0}\right) $ denote the monomials
defined by $e_{r}\left( x\right) =x^{r}$. Kumar and Deo explicitly
calculated the first 4 moments \cite[Lemma~1]{Kumar-Deo-Miskolc-2025} by
using the recursive formula \cite[Remark~1]{Kumar-Deo-Miskolc-2025} 
\begin{equation*}
n(P_{n}^{\beta }e_{r+1})\left( x\right) =x^{2}(P_{n}^{\beta }e_{r})^{\prime
}\left( x\right) +\left( \beta x^{2}+nx\right) \left( P_{n}^{\beta
}e_{r}\right) \left( x\right) ,
\end{equation*}%
which immediately follows from the differential equation $\left( \ref{ode-P}%
\right) $. For our purposes we need an explicit representation of all
moments $P_{n}^{\beta }e_{r}$ as a polynomial in the variable $1/n$. The
next lemma provides such an expression.

\begin{lemma}
\label{lemma-moments}For $r=0,1,2,\ldots $, the moments of the operators $%
P_{n}^{\beta }$ are given by 
\begin{equation*}
P_{n}^{\beta }e_{r}=\sum_{k=0}^{r}\frac{1}{n^{k}}\sum_{j=0}^{k}\frac{\beta
^{j}}{j!}e_{r+j}\sum_{i=0}^{k-j}\binom{j+i-1}{i}\frac{r!}{\left(
r-j-i\right) !}\left[ 
\begin{array}{c}
r-j-i \\ 
r-k%
\end{array}%
\right] .
\end{equation*}
\end{lemma}

For the proof we utilize the following binomial identity. For the sake of
completeness, we provide a short proof, though it is well-known.

\begin{lemma}
\label{lemma-binomial-identity}For real numbers $a,b$ and $m=0,1,2,\ldots $, 
\begin{equation*}
\binom{a+b-1}{m}=\sum_{i=0}^{m}\binom{a-1-i}{m-i}\binom{b-1+i}{i}.
\end{equation*}
\end{lemma}

\begin{proof}
Making use of the reflection formula and the Vandermonde identity, we obtain 
\begin{equation*}
\sum_{i=0}^{m}\binom{a-1-i}{m-i}\binom{b-1+i}{i}.=\left( -1\right)
^{m}\sum_{i=0}^{m}\binom{m-a}{m-i}\binom{-b}{i}=\left( -1\right) ^{m}\binom{%
m-a-b}{m}=\binom{a+b-1}{m}.
\end{equation*}
\end{proof}

\begin{proof}[Proof of Lemma~\protect \ref{lemma-moments}]
We have 
\begin{eqnarray*}
\left( P_{n}^{\beta }e_{r}\right) \left( x\right) &=&e^{-\beta x}\left( 
\frac{n}{x}\right) ^{n}\sum_{\nu =0}^{\infty }\frac{\left( n\beta \right)
^{\nu }}{\nu !}\frac{1}{\Gamma (n+\nu )}\int_{0}^{\infty }t^{n+\nu
-1+r}e^{-nt/x}dt \\
&=&e^{-\beta x}\left( \frac{x}{n}\right) ^{r}\sum_{\nu =0}^{\infty }\frac{%
\left( \beta x\right) ^{\nu }}{\nu !}\frac{\Gamma \left( n+\nu +r\right) }{%
\Gamma \left( n+\nu \right) }.
\end{eqnarray*}%
Observing that 
\begin{equation*}
\frac{\Gamma \left( n+\nu +r\right) }{\Gamma \left( n+\nu \right) }=\left.
\left( \left( \frac{d}{dz}\right) ^{r}z^{n+\nu +r-1}\right) \right \vert
_{z=1}
\end{equation*}%
we obtain 
\begin{equation*}
\left( P_{n}^{\beta }e_{r}\right) \left( x\right) =e^{-\beta x}\left( \frac{x%
}{n}\right) ^{r}\left. \left( \left( \frac{\partial }{\partial z}\right)
^{r}z^{n+r-1}e^{\beta xz}\right) \right \vert _{z=1}=\left( \frac{x}{n}%
\right) ^{r}\sum_{j=0}^{r}\binom{r}{j}\left( \beta x\right) ^{j}\frac{\Gamma
\left( n+r\right) }{\Gamma \left( n+j\right) },
\end{equation*}%
where we applied the Leibniz rule for differentiation. Applying Lemma~\ref%
{lemma-binomial-identity} with $a=n+r-j$ and $b=j$\ we infer that 
\begin{equation*}
\binom{r}{j}\frac{\Gamma \left( n+r\right) }{\Gamma \left( n+j\right) }=%
\frac{r!}{j!}\binom{n+r-1}{r-j}=\frac{r!}{j!}\sum_{i=0}^{r-j}\binom{n+r-j-i-1%
}{r-j-i}\binom{j+i-1}{i}.
\end{equation*}%
By $\left( \ref{Stirling-number-GF}\right) $, we have 
\begin{equation*}
\binom{n+r-j-i-1}{r-j-i}=\frac{1}{\left( r-j-i\right) !}\sum_{\ell
=0}^{r-j-i}\left[ 
\begin{array}{c}
r-j-i \\ 
r-j-i-\ell%
\end{array}%
\right] n^{r-j-i-\ell }.
\end{equation*}%
Thus, we arrive at 
\begin{equation*}
\left( P_{n}^{\beta }e_{r}\right) \left( x\right) =x^{r}\sum_{j=0}^{r}\frac{%
\left( \beta x\right) ^{j}}{j!}\sum_{i=0}^{r-j}\binom{j+i-1}{i}\frac{r!}{%
\left( r-j-i\right) !}\sum_{\ell =0}^{r-j-i}\left[ 
\begin{array}{c}
r-j-i \\ 
r-j-i-\ell%
\end{array}%
\right] n^{-j-i-\ell }.
\end{equation*}%
Collecting all terms with $j+i+\ell =k$\ we obtain the desired formula.
\end{proof}

Now we turn to the central moments of the operators $P_{n}^{\beta }$. For
each real number $x$, define the function $\psi _{x}$ by $\psi
_{x}=e_{1}-xe_{0}$.

\begin{lemma}
\label{lemma-central-moments}The central moments $\left( P_{n}^{\beta }\psi
_{x}^{s}\right) \left( x\right) $ of the operators $P_{n}^{\beta }$ are
given by $P_{n}^{\beta }\psi _{x}^{0}=e_{0}$ and, for $s\in \mathbb{N}$, by 
\begin{equation}
\left( P_{n}^{\beta }\psi _{x}^{s}\right) \left( x\right)
=x^{s}\sum_{k=\left \lfloor \left( s+1\right) /2\right \rfloor }^{s}\frac{1}{%
n^{k}}\sum_{j=0}^{k}\frac{\left( \beta x\right) ^{j}}{j!}a\left(
k,s,j\right) .  \label{central-moments-small-explicit representation}
\end{equation}%
with the coefficients $a\left( k,s,j\right) $ defined in $\left( \ref%
{def-coefficient-a(k,s,j)}\right) $.
\end{lemma}

As an immediate consequence of the preceding lemma we have the following
estimate of the central moments:\ For each $x>0$ and $s\in \mathbb{N}_{0}$,
the central moment $\left( P_{n}^{\beta }\psi _{x}^{s}\right) \left(
x\right) $ satisfies the asymptotic relation 
\begin{equation}
\left( P_{n}^{\beta }\psi _{x}^{s}\right) \left( x\right) =O\left(
n^{-\left \lfloor \left( s+1\right) /2\right \rfloor }\right) \text{ \qquad }%
\left( n\rightarrow \infty \right) .
\label{central-moments-asymptotic-relation}
\end{equation}%
Though the first four central moments were given in \cite[Lemma~2.]%
{Kumar-Deo-Miskolc-2025}, we present them for the convenience of the reader
in a compact form:\ 
\begin{eqnarray*}
\left( P_{n}^{\beta }\psi _{x}^{0}\right) \left( x\right) &=&1, \\
\left( P_{n}^{\beta }\psi _{x}^{1}\right) \left( x\right) &=&\frac{\beta
x^{2}}{n}, \\
\left( P_{n}^{\beta }\psi _{x}^{2}\right) \left( x\right) &=&\frac{x^{2}}{n}+%
\frac{\beta x^{3}\left( 2+\beta x\right) }{n^{2}}, \\
\left( P_{n}^{\beta }\psi _{x}^{3}\right) \left( x\right) &=&\frac{%
x^{3}\left( 2+3\beta x\right) }{n^{2}}+\frac{\beta x^{4}\left( 6+6\beta
x+\beta ^{2}x^{2}\right) }{n^{3}}, \\
\left( P_{n}^{\beta }\psi _{x}^{4}\right) \left( x\right) &=&\frac{3x^{4}}{%
n^{2}}+\frac{2x^{4}\left( 3+\beta x\right) \left( 1+3\beta x\right) }{n^{3}}+%
\frac{\beta x^{5}\left( 24+\beta x\left( 6+\beta x\right) ^{2}\right) }{n^{4}%
}.
\end{eqnarray*}

\begin{proof}[Proof of Lemma~\protect \ref{lemma-central-moments}]
By the binomial formula, we obtain 
\begin{equation*}
\left( P_{n}^{\beta }\psi _{x}^{s}\right) \left( x\right)
=\sum_{r=0}^{s}\left( -x\right) ^{s-r}\binom{s}{r}\left( P_{n}^{\beta
}e_{r}\right) \left( x\right) .
\end{equation*}%
Lemma~\ref{lemma-moments} leads to 
\begin{eqnarray*}
\left( P_{n}^{\beta }\psi _{x}^{s}\right) \left( x\right)
&=&x^{s}\sum_{k=0}^{s}\frac{1}{n^{k}}\sum_{j=0}^{k}\frac{\left( \beta
x\right) ^{j}}{j!}\sum_{i=0}^{k-j}\binom{j+i-1}{i} \\
&&\times \sum_{r=k}^{s}\left( -1\right) ^{s-r}\binom{s}{r}\frac{r!}{\left(
r-j-i\right) !}\left[ 
\begin{array}{c}
r-j-i \\ 
r-k%
\end{array}%
\right] .
\end{eqnarray*}%
The identity $\binom{s}{r}\frac{r!}{\left( r-j-i\right) !}=\frac{s!}{\left(
s-j-i\right) !}\binom{s-j-i}{r-j-i}$ implies that 
\begin{equation*}
\left( P_{n}^{\beta }\psi _{x}^{s}\right) \left( x\right)
=x^{s}\sum_{k=0}^{s}\frac{1}{n^{k}}\sum_{j=0}^{k}\frac{\left( \beta x\right)
^{j}}{j!}\sum_{i=0}^{k-j}\binom{j+i-1}{i}\frac{s!}{\left( s-j-i\right) !}%
C\left( k,s,j,i\right) ,
\end{equation*}%
where 
\begin{equation*}
C\left( k,s,j,i\right) =\sum_{r=k}^{s}\left( -1\right) ^{s-r}\binom{s-j-i}{%
r-j-i}\left[ 
\begin{array}{c}
r-j-i \\ 
r-k%
\end{array}%
\right] .
\end{equation*}%
By representation $\left( \ref{associated Stirling numbers of the first kind}%
\right) $, we have, for $0\leq j+i\leq k\leq r$, 
\begin{equation*}
\left[ 
\begin{array}{c}
r-j-i \\ 
r-k%
\end{array}%
\right] =\sum_{\ell =0}^{k-j-i}s_{2}\left( k-j-i+\ell ,\ell \right) \binom{%
r-j-i}{k-j-i+\ell }.
\end{equation*}%
Using $\binom{r-j-i}{k-j-i+\ell }=0$, for $j+i\leq r<k-\ell $, and the
identity $\binom{s-j-i}{r-j-i}\binom{r-j-i}{k-j-i+\ell }=\binom{s-j-i}{%
k-j-i+\ell }\binom{s-k-\ell }{r-k-\ell }$ we obtain 
\begin{equation*}
C\left( k,s,j,i\right) =\sum_{\ell =0}^{k-j-i}s_{2}\left( k-j-i+\ell ,\ell
\right) \binom{s-j-i}{k-j-i+\ell }\sum_{r=k+\ell }^{s}\left( -1\right) ^{s-r}%
\binom{s-k-\ell }{r-k-\ell }.
\end{equation*}%
Observing that 
\begin{equation*}
\sum_{r=k+\ell }^{s}\left( -1\right) ^{s-r}\binom{s-k-\ell }{r-k-\ell }%
=\sum_{r=0}^{s-k-\ell }\left( -1\right) ^{s-k-\ell -r}\binom{s-k-\ell }{r}%
=\left \{ 
\begin{tabular}{lll}
$0$ &  & $\left( k+\ell <s\right) ,$ \\ 
&  &  \\ 
$1$ &  & $\left( k+\ell =s\right) $%
\end{tabular}%
\right.
\end{equation*}%
we infer that 
\begin{equation*}
C\left( k,s,j,i\right) =\left \{ 
\begin{tabular}{lll}
$0$ &  & $\left( s-k>k-j-i\right) ,$ \\ 
&  &  \\ 
$s_{2}\left( s-j-i,s-k\right) $ &  & $\left( s-k\leq k-j-i\right) .$%
\end{tabular}%
\right.
\end{equation*}%
Hence, we obtain 
\begin{equation*}
\left( P_{n}^{\beta }\psi _{x}^{s}\right) \left( x\right)
=x^{s}\sum_{k=0}^{s}\frac{1}{n^{k}}\sum_{j=0}^{k}\frac{\left( \beta x\right)
^{j}}{j!}a\left( k,s,j\right) ,
\end{equation*}%
where 
\begin{equation*}
a\left( k,s,j\right) =\sum_{i=0}^{\min \left \{ k-j,2k-s-j\right \} }\binom{%
j+i-1}{i}\frac{s!}{\left( s-j-i\right) !}s_{2}\left( s-j-i,s-k\right)
\end{equation*}%
with the convention that a sum is to be read as zero if the lower index
exceeds the upper index. An index shift replacing $i$ with $i-j$ yields 
\begin{equation*}
a\left( k,s,j\right) =\sum_{i=j}^{\min \left \{ k,2k-s\right \} }\binom{i-1}{%
i-j}\frac{s!}{\left( s-i\right) !}s_{2}\left( s-i,s-k\right) .
\end{equation*}%
Finally, we conclude that $a\left( k,s,j\right) =0$ if $2k<s$. The latter
condition is equivalent to $k<\left \lfloor \left( s+1\right) /2\right \rfloor 
$.
\end{proof}

In the proof of the main result we apply a general approximation theorem of
Sikkema \cite[Theorem~3]{Sikkema-1970}. Since it deals with functions of
polynomial growth, we cannot apply it directly to functions $f\in E\left[
0,\infty \right) $. First, we note that the definition of $E\left[ 0,\infty
\right) $ implies that the function is bounded on each finite interval $%
\left[ 0,r\right] $ , $r>0$. In order to prove Theorem~\ref%
{theorem-expansion} for functions of (at most) exponential growth, we need a
localization result (Lemma~\ref{lemma-localization}) for the operators $%
P_{n}^{\beta }$. The next result will be applied in its proof.

For real numbers $A$, define $\exp _{A}\left( t\right) =e^{At}$.

\begin{lemma}
\label{lemma-exp}Let $x>0$. Then $\left( P_{n}^{\beta }\exp _{A}\right)
\left( x\right) $ is well defined, for $n>Ax$, and has the limit $%
\lim_{n\rightarrow \infty }\left( P_{n}^{\beta }\exp _{A}\right) \left(
x\right) =\exp _{A}\left( x\right) $.
\end{lemma}

\begin{proof}
Let $x>0$ and $n>Ax$. An elementary calculation shows that 
\begin{equation*}
\left( P_{n}^{\beta }\exp _{A}\right) \left( x\right) =\exp \left( \beta x%
\frac{Ax}{n-Ax}\right) \left( 1-\frac{Ax}{n}\right) ^{-n}\rightarrow \exp
_{A}\left( x\right)
\end{equation*}
as $n\rightarrow \infty $.
\end{proof}

\begin{lemma}
\label{lemma-localization}(Localization theorem) Let $\delta >0$ and fix $%
x>0 $. If a function $f\in E_{A}\left[ 0,\infty \right) $ vanishes on the
interval $\left( x-\delta ,x+\delta \right) \cap \left( 0,\infty \right) $,
then $\left( P_{n}^{\beta }f\right) \left( x\right) =O\left( n^{-m}\right) $
as $n\rightarrow \infty $,\ for arbitrarily large $m>0$.
\end{lemma}

\begin{proof}
Let $m\in \mathbb{N}$. For a certain positive constant $C$, we have 
\begin{eqnarray*}
\left \vert \left( P_{n}^{\beta }f\right) \left( x\right) \right \vert &\leq
&C\int_{\left( 0,\infty \right) \setminus \left( x-\delta ,x+\delta \right)
}W_{\beta }\left( n,x,t\right) e^{At}dt \\
&\leq &C\delta ^{-2m}\int_{0}^{\infty }W_{\beta }\left( n,x,t\right) \left(
t-x\right) ^{2m}e^{At}dt.
\end{eqnarray*}%
Application of the Schwarz inequality yields 
\begin{equation*}
\left \vert \left( P_{n}^{\beta }f\right) \left( x\right) \right \vert \leq
C\delta ^{-2m}\sqrt{\left( P_{n}^{\beta }\exp _{2A}\right) \left( x\right) }%
\sqrt{\left( P_{n}^{\beta }\psi _{x}^{4m}\right) \left( x\right) }.
\end{equation*}%
By Lemma~\ref{lemma-exp}, the first root has the finite limit $\exp
_{A}\left( x\right) $ as $n\rightarrow \infty $. By relation~$\left( \ref%
{central-moments-asymptotic-relation}\right) $, the second root satisfies $%
\sqrt{\left( P_{n}^{\beta }\psi _{x}^{4m}\right) \left( x\right) }=O\left(
n^{-m}\right) $ as $n\rightarrow \infty $. This completes the proof.
\end{proof}

\begin{proof}[Proof of Theorem~\protect \ref{theorem-expansion}]
Let $x\in \left( 0,\infty \right) $ and $q\in \mathbb{N}$. Suppose that the
function $f\in E\left[ 0,\infty \right) $ admits the derivative $f^{\left(
2q\right) }\left( x\right) $. We assume that $f\in E_{A}\left[ 0,\infty
\right) $, for some $A>0$. For a certain $\delta >0$, let $\widetilde{f}$ be
the function such that $\widetilde{f}=f$ on $\left( x-\delta ,x+\delta
\right) \cap \left( 0,\infty \right) $ and $\widetilde{f}\left( t\right) =0$
otherwise. Obviously, $\widetilde{f}\in E_{A}\left[ 0,\infty \right) $. By
relation~$\left( \ref{central-moments-asymptotic-relation}\right) $, we have 
$\left( P_{n}^{\beta }\psi _{x}^{2s}\right) \left( x\right) =O\left(
n^{-s}\right) $ as $n\rightarrow \infty $. Therefore, the general
approximation theorem due to Sikkema \cite[Theorem~3]{Sikkema-1970} implies
that 
\begin{equation*}
\left( P_{n}^{\beta }\widetilde{f}\right) \left( x\right) =\sum_{s=0}^{2q}%
\frac{1}{s!}\widetilde{f}^{\left( s\right) }\left( x\right) \left(
P_{n}^{\beta }\psi _{x}^{s}\right) \left( x\right) +o\left( n^{-q}\right) 
\text{ \qquad }\left( n\rightarrow \infty \right) .
\end{equation*}%
Application of the localization theorem (Lemma~\ref{lemma-localization})
with $m=q+1$ yields $\left( P_{n}^{\beta }\left( f-\widetilde{f}\right)
\right) \left( x\right) =O\left( n^{-q-1}\right) $ as $n\rightarrow \infty $%
. In view of $P_{n}^{\beta }f=P_{n}^{\beta }\widetilde{f}+P_{n}^{\beta
}\left( f-\widetilde{f}\right) $, we infer that $\left( P_{n}^{\beta
}f\right) \left( x\right) =\left( P_{n}^{\beta }\widetilde{f}\right) \left(
x\right) +o\left( n^{-q}\right) $ as $n\rightarrow \infty $. Because $%
\widetilde{f}$ and $f$ coincide in the interval $\left( x-\delta ,x+\delta
\right) $ we have $\widetilde{f}^{\left( s\right) }\left( x\right)
=f^{\left( s\right) }\left( x\right) $, for $s=0,\ldots ,2q$. Hence, both
functions satisfy the same asymptotic relation, i.e., 
\begin{equation*}
\left( P_{n}^{\beta }f\right) \left( x\right) =\sum_{s=0}^{2q}\frac{%
f^{\left( s\right) }\left( x\right) }{s!}\left( P_{n}^{\beta }\psi
_{x}^{s}\right) \left( x\right) +o\left( n^{-q}\right) \text{ }\qquad \left(
n\rightarrow \infty \right) .
\end{equation*}%
\newline
By Lemma~\ref{lemma-central-moments} and Eq.~$\left( \ref%
{central-moments-small-explicit representation}\right) $, and noting that $%
k\geq \left \lfloor \left( s+1\right) /2\right \rfloor $ is equivalent to $%
s\leq 2k$, we infer that 
\begin{eqnarray*}
&&\sum_{s=0}^{2q}\frac{f^{\left( s\right) }\left( x\right) }{s!}%
x^{s}\sum_{k=\left \lfloor \left( s+1\right) /2\right \rfloor }^{s}\frac{1}{%
n^{k}}\sum_{j=0}^{k}\frac{\left( \beta x\right) ^{j}}{j!}a\left( k,s,j\right)
\\
&=&\sum_{k=0}^{2q}\frac{1}{n^{k}}\sum_{s=k}^{2k}\frac{f^{\left( s\right)
}\left( x\right) }{s!}x^{s}\sum_{j=0}^{k}\frac{\left( \beta x\right) ^{j}}{j!%
}a\left( k,s,j\right) .
\end{eqnarray*}%
This completes the proof.
\end{proof}


\strut

\thispagestyle{empty}

~\vfill

\end{document}